\documentclass[11pt]{amsart}

\author[C.~Sanna]{Carlo Sanna$^\dagger$}
\thanks{$\dagger\,$C.~Sanna is a member of the INdAM group GNSAGA}
\address{\parbox{\linewidth}{
Politecnico di Torino, Department of Mathematical Sciences\\
Corso Duca degli Abruzzi 24, 10129 Torino, Italy\\[-8pt]}}
\email{carlo.sanna.dev@gmail.com}

\title{Practical central binomial coefficients}

\usepackage{amsmath}
\usepackage{amssymb}
\usepackage{amsthm}
\usepackage{geometry}
\usepackage{color}
\usepackage{hyperref}
\usepackage{enumerate}
\hypersetup{colorlinks=true}

\newtheorem{thm}{Theorem}[section]
\newtheorem{cor}{Corollary}[section]
\newtheorem{lem}[thm]{Lemma}
\theoremstyle{remark}

\begin{document}

\begin{abstract}
A \emph{practical number} is a positive integer $n$ such that all positive integers less than $n$ can be written as a sum of distinct divisors of $n$.
Leonetti and Sanna proved that, as $x \to +\infty$, the central binomial coefficient $\binom{2n}{n}$ is a practical number for all positive integers $n \leq x$ but at most $O(x^{0.88097})$ exceptions.
We improve this result by reducing the number of exceptions to $\exp\!\big(C (\log x)^{4/5} \log \log x\big)$, where $C > 0$ is a constant.
\end{abstract}

\maketitle

\section{Introduction}

A \emph{practical number} is a positive integer $n$ such that all positive integers less than $n$ can be written as a sum of distinct divisors of $n$.
Practical numbers were defined by Srinivasan~\cite{MR0027799}, althought they were already used by Fibonacci to decompose rational numbers as sums of unit fractions~\cite[pag.~121]{MR1923794}.
Estimates for the counting function of practical numbers were given by Hausman and Shapiro~\cite{MR752596}, Tenenbaum~\cite{MR860809}, Margenstern~\cite{MR1089787}, Saias~\cite{MR1430008}, and, lastly, Weingartner~\cite{MR3356847}, who proved that the number of practical numbers up to $x$ is asymptotic to $c x / \log x$, as $x \to +\infty$, where $c = 1.33607\dots$~\cite{MR4079398}, settling a conjecture of Margenstern~\cite{MR1089787}.

In analogy with Goldbach's conjecture and prime triplet conjecture, Melfi~\cite{MR1370203} proved that every positive even integer is the sum of two practical numbers, and that there are infinitely many triples $(n, n + 2, n + 4)$ of practical numbers.
Moreover, Melfi~\cite{MR1452391} proved that every Lucas~sequence $(U_n(P,Q))$ satisfying some mild conditions contains infinitely many practical numbers, and Sanna~\cite{MR4003803} showed that $U_n(P,Q)$ is practical for at least $\gg_{P,Q} x / \log x$ positive integers $n \leq x$, as $x \to +\infty$; and asked for a nontrivial upper bound.

Leonetti and Sanna~\cite{MR4017943} studied binomial coefficients that are practical numbers.
They proved that, for fixed $\varepsilon > 0$ and as $x \to +\infty$, all binomial coefficients $\binom{n}{k}$, with $0 \leq k \leq n \leq x$, are practical numbers but at most $O_\varepsilon\!\big(x^{2 - (2^{-1}\!\log 2 - \varepsilon) / \log \log x}\big)$ exceptions.
Furthermore, they showed that the central binomial coefficient $\binom{2n}{n}$ is a practical number for all positive integers $n \leq x$ but at most $O(x^{0.88097})$ exceptions.
In this note, we give the following improvement of the last result.

\begin{thm}\label{thm:main}
For $x \geq 3$ the central binomial coefficient $\binom{2n}{n}$ is a practical number for all positive integers $n \leq x$ but at most $\exp\!\big(C (\log x)^{4/5} \log \log x\big)$ exceptions, where $C > 0$ is a constant.
\end{thm}

We remark that (as already pointed out in~\cite{MR4017943}), likely, there are only finitely many positive integers $n$ such that $\binom{2n}{n}$ is not a practical number, but proving so could be out of reach.
In~fact, if $n$ is a power of $2$ whose base $3$ representation does not contain the digit $2$, then $\binom{2n}{n}$ is not a practical number~\cite[Proposition 2.1]{MR4017943}.
However, establishing whether there are finitely or infinitely many such powers of $2$ is an open problem~\cite{MR580438, MR1845703, MR2506687, MR623247}.

\section{Preliminaries}

We need some preliminary results.

\begin{lem}\label{lem:2d}
If $d$ is a practical number and $n$ is a positive integer divisible by $d$ and having all prime factors not exceeding $2d$, then $n$ is a practical number.
\end{lem}
\begin{proof}
See~\cite[Lemma~2.2]{MR4017943}.
\end{proof}

For every positive integer $n$, let $s_2(n)$ be the number of nonzero binary digits of $n$.

\begin{lem}\label{lem:v2}
For every positive integer $n$, the exponent of $2$ in the prime factorization of $\binom{2n}{n}$ is equal to $s_2(n)$.
\end{lem}
\begin{proof}
A result of Kummer~\cite{MR1578793} says that for every prime number $p$ and for all positive integers $m, n$ the exponent of $p$ in the prime factorization of $\binom{m + n}{n}$ is equal to the number of carries in the addition $m + n$ done in base $p$.
If $m = n$ and $p = 2$ then we get the desired claim.
\end{proof}

\begin{lem}\label{lem:s2}
We have
\begin{equation*}
\#\big\{n \leq x : s_2(n) \leq \varepsilon (\log n / \log 2 + 1)\big\} \leq x^{\left(\tfrac1{\log 2} + o(1)\right) \, \varepsilon \log (1 / \varepsilon)} ,
\end{equation*}
uniformly as $\varepsilon \log x \to +\infty$ and $\varepsilon \to 0^+$.
\end{lem}
\begin{proof}
Put $N := \lfloor \log x / \log 2 + 1 \rfloor$ and $k := \lceil \varepsilon (\log n / \log 2 + 1) \rceil$.
Then
\begin{equation*}
C := \#\big\{n \leq x : s_2(n) \leq \varepsilon (\log n / \log 2 + 1)\big\} \leq \#\big\{n < 2^N : s_2(n) \leq k \big\} ,
\end{equation*}
where the right-hand side is the number of binary strings of length $N$ having at most $k$ nonzero bits (including $n=0$ to the count).
Therefore,
\begin{align*}
C &\leq \sum_{j \,=\, 0}^k \binom{N}{j} \leq \sum_{j \,=\, 0}^k \frac{N^j}{j!} = \sum_{j \,=\, 0}^k \frac{k^j}{j!}\left(\frac{N}{k}\right)^j < \left(\frac{eN}{k}\right)^k < e^{(1-\log \varepsilon)(\varepsilon(\log x / \log 2 + 1) + 1)} ,
\end{align*}
and the claim follows recalling that $\varepsilon \log x \to +\infty$ and $\varepsilon \to 0^+$.
\end{proof}

The following result of Erd\H{o}s and Kolesnik is the key to the proof of Theorem~\ref{thm:main}.

\begin{thm}\label{thm:EK}
There exist constants $c_1, c_2 > 0$ such that, for all integers $m, n, r$ with
\begin{equation*}
2 \leq m \leq n / 2 \quad \text{ and } \quad 1 \leq r \leq c_1\!\left(\frac{(\log m)^3}{(\log n)^2 \log \log n}\right)^{1/4} ,
\end{equation*}
there exist at least $c_2 r m^{1/r} / (4^r \log m)$ prime numbers $p \in {[m^{1/r}, n^{1/r}]}$ such that $p^r \mid\mid {\textstyle\binom{n}{m}}$.
\end{thm}
\begin{proof}
See~\cite[Theorem~2]{MR1692284}.
\end{proof}

\begin{cor}\label{cor:EK}
There exists a constant $c_3 > 0$ such that, for all integers $n, r$ with
\begin{equation*}
n \geq 3 \quad \text{ and } \quad 1 \leq r \leq c_3\!\left(\frac{\log n}{\log \log n}\right)^{1/4} ,
\end{equation*}
there exists a prime number $p \in {[n^{1/r}, (2n)^{1/r}]}$ such that $p^r \mid\mid \binom{2n}{n}$.
\end{cor}
\begin{proof}
The claim follows by replacing $m$ and $n$ with $n$ and $2n$, respectively, in Theorem~\ref{thm:EK}.
\end{proof}

\section{Proof of Theorem~\ref{thm:main}}

Fix $C > \max\!\big((5 \log 2)^{-1}, (2 / c_3)^4\big)$, where $c_3$ is the constant of Corollary~\ref{cor:EK}.
Assume that $x$ is sufficiently large and put $E := \exp\!\big(C(\log x)^{4/5} \log \log x\big)$ and $\varepsilon := (\log x)^{-1/5}$.
Let $n \leq x$ be a positive integer and let $v$ be the exponent of $2$ in the prime factorization of $\binom{2n}{n}$.
Since
\begin{equation*}
\frac1{\log2}\, \varepsilon \log (1/\varepsilon) \log x = \frac1{5 \log 2} (\log x)^{4/5} \log \log x < C (\log x)^{4/5} \log \log x ,
\end{equation*}
from Lemma~\ref{lem:v2} and Lemma~\ref{lem:s2} we get that $2^v \leq n^\varepsilon$ for less than $\tfrac1{2}E$ choices of $n$.
Hence, we can assume that $2^v > n^\varepsilon$ and $n > \tfrac1{2}E$, which excludes at most $E$ positive integers not exceeding $x$.
Then, since $n > \tfrac1{2} E$ and $x$ is sufficiently large, we have
\begin{equation*}
\frac{\log n}{\log \log n} > \frac{\log (\tfrac1{2}E)}{\log \log (\tfrac1{2}E)} > C (\log x)^{4/5} > \left(\frac{2(\log x)^{1/5}}{c_3}\right)^4 .
\end{equation*}
Therefore,
\begin{equation*}
r := \left\lfloor c_3 \left(\frac{\log n}{\log \log n}\right)^{1/4}\right\rfloor > \frac1{\varepsilon} .
\end{equation*}
Thanks to Corollary~\ref{cor:EK}, there exists a prime number $p \in {[n^{1/r}, (2n)^{1/r}]}$ such that $p^r$ divides $\binom{2n}{n}$.
Now $2^v$ is a practical number, because all powers of $2$ are practical numbers.
Morever, since
\begin{equation*}
p \leq (2n)^{1/r} < (2n)^\varepsilon < 2^{v + 1} ,
\end{equation*}
from Lemma~\ref{lem:2d} it follows that $2^v p^r$ is a practical number.
Finally, $2^v p^r$ divides $\binom{2n}{n}$, $2^v p^r \geq 2 n$, and all prime factors of $\binom{2n}{n}$ are not exceeding $2n$, hence Lemma~\ref{lem:2d} yields that $\binom{2n}{n}$ is a practical number.
The proof is complete.

\bibliographystyle{amsplain}

\end{document}